\documentclass[a4paper,12pt]{amsart} 

\usepackage{amssymb,amsmath}
\usepackage{pb-diagram}
\usepackage{pb-xy}
\usepackage[cmtip,arrow]{xy}

\newtheorem{thm}{Theorem}[section] 
\newtheorem*{thm*}{Theorem}

\newtheorem{cor}[thm]{Corollary} 
\newtheorem{lem}[thm]{Lemma}

\theoremstyle{definition}

\newtheorem{exa}[thm]{Example}

\numberwithin{equation}{section}

\newcommand{\skal}[2]{\langle #1,#2\rangle}
\newcommand{\alg}[1]{\mathfrak{#1}}

\begin{document}

\title{Rodin's formula in arbitrary codimension}
\author{Ma\l gorzata Ciska-Niedzia\l omska}
\author{Kamil Niedzia\l omski}
\date{\today}

 
\address{
Department of Mathematics and Computer Science \endgraf
University of \L\'{o}d\'{z} \endgraf
ul. Banacha 22, 90-238 \L\'{o}d\'{z} \endgraf
Poland
}
\email{mciska@math.uni.lodz.pl}
\email{kamiln@math.uni.lodz.pl}

\begin{abstract}
We extend the Rodin's formula for $p$--modulus of the family of curves in $\mathbb{R}^n$ to arbitrary codimension. The proof relies on the formula for the $p$--modulus of family of level sets of a submersion and an algebraic lemma relating Jacobi matrices of considered maps. We state appropriate examples.
\end{abstract}

\maketitle

\section{Introduction}

The aim of this note is to extend to any codimension the Rodin's formula for $2$--modulus of family of curves \cite{Rodin} and its generalization to any $p>1$ obtained recently by Brakalova, Markina and Vasil'ev \cite{BMV}. We rely on the formula for the $p$--modulus and the extremal function for the family of surfaces given by the level sets of a submersion \cite{KP}. The key observation is the relation between appropriate jacobians. This is purely algebraic fact.  We conclude by stating appropriate examples.

Let us introduce the necessary notions and give an outline of the approach for codimension $n-1$ case. 

\subsection{Fuglede $p$--modulus}
The $p$--modulus, introduced by Fuglede in \cite{BF}, is a powerful tool in geometric measure theory, especially in the context of quasiconformal maps, weak upper gradients, harmonicity on metric measure spaces, etc. 

Consider a family $\Sigma$ of measures on a measure space $(X,\alg{m})$ such that each $\mu$--measurable function for all $\mu\in \Sigma$ is $\alg{m}$--measurable. Then, we say that such non--negative $f$ is {\it admissible} if 
\begin{equation}\label{eq:admissible}
\int_X f\,d\mu\geq 1\quad\textrm{for all $\mu\in\Sigma$}.
\end{equation}
We write $f\in{\rm adm}(\Sigma)$. By the $p$--modulus of $\Sigma$ (with respect to $\alg{m}$), where $p>1$ is fixed, we mean a number
\begin{equation}\label{eq:pmodulus}
{\rm mod}_p(\Sigma)=\inf_{f\in{\rm adm}(\Sigma)}\int_X f^p\,d\alg{m}.
\end{equation}
If ${\rm adm}(\Sigma)$ is empty we put ${\rm mod}_p(\Sigma)=\infty$. It can be shown that the assignment $\Sigma\to {\rm mod}_p(\Sigma)$ is an outer measure on the space of measures whose $\sigma$--algebras are contained in $\sigma$--algebra of $\alg{m}$. Moreover, we say that an admissible function $f_{\Sigma}$ is {\it extremal} if ${\rm mod}_p(\Sigma)=\int_X f_{\Sigma}^p\,d\alg{m}$. It is an easy consequence of properties of $p$--modulus \cite{BF}, that up to a subfamily of $p$--modulus zero, there is unique extremal function.

In our opinion, the key features of this notion are the following:					
\begin{itemize}
\item for a family of $k$--dimensional surfaces, or more precisely, the Hausdorff measures associated to them, th $p$--modulus is conformal invariant provided $n=kp$,
\item if $n$--modulus of family $\Sigma$ of curves is non--zero, then it is nonzero under quasi--conformal deformation of $\mathbb{R}^n$. More precisely, if $f:\Omega\to \Omega'$ is a $K$--quasiconformal map, then
\begin{equation*}
\frac{1}{K}{\rm mod}_n(f(\Sigma))\leq{\rm mod}_n(\Sigma)\leq K{\rm mod}_n(f(\Sigma)), 
\end{equation*} 
\item if $f$ is $p$--integrable with respect to $\alg{m}$, then, up to a subfamily of zero $p$--modulus, $f$ is integrable with respect to any measure in $\Sigma$,
\item if $(f_n)$ is a sequence, which converges in $L^p(X,\alg{m})$ to $f_0$, then, up to a subsequence and a family of $p$--modulus equal zero, $(f_n)$ converges to $f_0$ in $L^1(X,\mu)$ for $\mu\in\Sigma$.
\end{itemize}

\subsection{Alternative approach to Rodin's formula} 
The second fact stated in the list above found many applications in the study of geometry of quasi--conformal maps (for example, see \cite{A}). In fact it can be considered as a definition of quasi--conformality. Thus it desirable to be able to calculate $n$--modulus of families of curves and, in particular, $2$--modulus for plane curves. Such formula was obtained by Rodin \cite{Rodin}.

\begin{thm}[Rodin \cite{Rodin}]\label{thm:Rodin}
Let $f:[0,1]\times[0,b]\to Q\subset\mathbb{R}^2$ be a smooth orientation preserving homeomorphism such that its jacobian $J_f$ is strictly positive. Let $\Gamma_0$ be a family of vertical lines in $[0,1]\times[0,b]$ and denote by $\Gamma$ the family of curves being the images of these lines with respect to $f$. Then the extremal function $f_{\Gamma}$ for the $2$--modulus of $\Gamma$ equals
\begin{equation}\label{eq:Rodinextremal}
f_{\Gamma}(z)=\frac{1}{l(x)}\left(\frac{|\dot{c}_x|}{J_f}\right)\circ f^{-1}(z),\quad f(x,t)=z,
\end{equation}
where $c_x(t)=f(x,t)$, dot denotes the derivative with respect to $t$ and
\begin{equation}\label{eq:Rodinhat}
l(x)=\int_0^b \frac{|\dot{c}_x|^2}{J_f}\,dt.
\end{equation}
Moreover, the $2$--modulus of $\Gamma$ equals
\begin{equation*}
{\rm mod}_2(\Gamma)=\int_0^1 \frac{1}{l(x)}\,dx.
\end{equation*}
\end{thm}

Let us rewrite formulae \eqref{eq:Rodinextremal} and \eqref{eq:Rodinhat} in a slightly different form, which will be more adequate for further considerations. Firstly, notice that
\begin{equation}\label{eq:Rodinhat2}
l(x)=\int_{c_x}\frac{|\dot{c_x}|}{J_f}\circ f^{-1}\,d H^1_{c_x},
\end{equation}
where $H^1_{c_x}$ is a volume element for $c_x$ (the $1$--dimensional Hausdorff measure on $c_x$). Denoting $l(x)$ by $\widehat{\left({\frac{|\dot{c_x}|}{J_f}}\right)}$, i.e., the integral over a curve $c_x$, formula \eqref{eq:Rodinextremal} takes the form
\begin{equation}\label{eq:Rodinextremal2}
f_{\Gamma}(z)=\frac{1}{\widehat{\left(\frac{|\dot{c}_x|}{J_f}\right)}}
\left(\frac{|\dot{c}_x|}{J_f}\right)\circ f^{-1}(z).
\end{equation}
From this representation, it is clear, than the integral of the extremal function $f_{\Gamma}$ on any curve from the family $\Gamma$ equals one (see also \cite{MCN} for more general approach). Let us now concentrate on the quantity $\frac{|\dot{c}_x|}{J_f}$. Put
\begin{equation*}
F(z)=x,\quad f(x,t)=z.
\end{equation*}
In other words, $F=\pi\circ f^{-1}$, where $\pi:[0,1]\times[0,b]\to [0,1]$ is a projection onto the first factor, $\pi(x,t)=x$. Then $F$ is a submersion with the fibers being the curves $c_x$ in $\Gamma$, $F(c_x(t))=x$. Differentiating this relation we have
\begin{equation*}
\skal{\nabla F}{D_xf}=1\quad\textrm{and}\quad\skal{\nabla F}{D_yf}=0,
\end{equation*}
where $D_xf$ and $D_yf$ denote differentials of $f$ with respect to $x$ and $y$, respectively, and $\skal{u}{v}=u_1v_1+u_2v_2$, $u=(u_1,u_2)$, $v=(v_1,v_2)$, is a standard inner product in $\mathbb{R}^2$. The key relation 
\begin{equation}\label{eq:mainalgebraic1}
\frac{|\dot{c_x}|}{J_f}=|\nabla F|\circ f,
\end{equation}
follows by the following simple algebraic fact applied to $A=Df$ and $v=(\nabla F)\circ f$.
\begin{lem}\label{lem:algebraic1}
Let $A\in GL(2,\mathbb{R})$ and let $v$ be any non--zero vector in $\mathbb{R}^2$. Denote the columns of $A$ by $u$ and $w$. Assuming $\skal{v}{u}=1$ and $\skal{v}{w}=0$ we have
\begin{equation*}
|w|=|\det A||v|.
\end{equation*} 
\end{lem}
\begin{proof}
Conditions $\skal{v}{u}=1$ and $\skal{v}{w}=0$ define a system of two linear equations with the unknown vector $v$. It has a unique solution, which is a first row of the inverse matrix $A^{-1}$, i.e.
\begin{equation*}
v=\frac{1}{\det A}(w_2,-w_1),\quad\textrm{where $w=(w_1,w_2)$}.
\end{equation*} 
Hence $|v|=\frac{1}{|\det A|}|w|$, which is a required relation.
\end{proof}
Applying \eqref{eq:mainalgebraic1}, conditions \eqref{eq:Rodinhat2} and \eqref{eq:Rodinextremal2} are equivalent to 
\begin{equation*}
l(x)=\int_{c_x}|\nabla F|\,dH^1_{c_x}=\widehat{|\nabla F|}\quad\textrm{and}\quad
f_{\Gamma}(z)=\frac{1}{\widehat{|\nabla F|}}\nabla F\circ f^{-1}(z),
\end{equation*}
respectively, but this is just the formula for the extremal function for the $p$--modulus of family of curves given by a submersion \cite{KP} (see also the following section).

In the case of arbitrary $p>1$ and arbitrary dimension $n$ we have the result by Bralkova, Markina, Vasil'ev \cite{BMV}
\begin{thm}[Rodin's generalized formula \cite{BMV}]\label{thm:BMV}
Let $f$ belong to the Sobolev space $W^{1,p}(U,\mathbb{R}^n)$ with positive jacobian $J_f$ for almost all points in $U$, where $U$ is a connected neighborhood of a product $D\times[a,b]$ and $D$ is compact in $\mathbb{R}^{n-1}$. Let $\Gamma$ be a family of curves $\{c_x\}_{x\in D}$ in $\Omega=f(U)$, where $c_x(t)=f(x,t)$. Then the extremal function for the $p$--modulus of $\Gamma$ is given by the formula
\begin{equation*}
f_{\Gamma}(z)=\frac{1}{l(x)}\left(\frac{|\dot{c}_x|}{J_f}\right)^{q-1}\circ f^{-1}(z),\quad z=f(x,t)\in\Omega,
\end{equation*}
where
\begin{equation*}
l(x)=\int_a^b \left(\frac{|\dot{c}_x|}{J_f}\right)^q J_f\,dt.
\end{equation*} 
($q$ is conjugate to $p$, i.e., $\frac{1}{p}+\frac{1}{q}=1$). Moreover, the $p$--modulus of $\Gamma$ equals
\begin{equation*}
{\rm mod}_p(\Gamma)=\int_D l(x)^{1-p}\,dx.
\end{equation*}
\end{thm}

In this case homeomorphism $f$ maps straight lines parametrized by some compact set $D$. The version stated here slightly differs from the original stated in \cite{BMV}, where the authors consider $f\circ u$ instead of $f$, $u$ being $C^1$--smooth diffeomorphism defining a condenser (see the following sections for more details). Such approach was probably chosen by the authors of \cite{BMV} in order to compute modulus of families of curves and its diffeomorphic images.

\subsection{Statement of the result}
In this article, we generalize Theorem \ref{thm:BMV} to arbitrary codimension. We rely on the formula for the $p$--modulus of a family of level sets of a submersion proved (in the smooth category) in \cite{KP}. We could consider the approach analogous to the one used in \cite{BMV}, where the authors apply the result of Badger \cite{Badger} for the extremal function. However, we want to keep the article as elementary as possible and, additionally, give an alternative proof of known results. Since, we are not mainly interested in the $p$--modulus of families separating plates of condenser or joining these plates and we focus on the general formula for the $p$--modulus we consider parametrizations (as usual) defined on open domains. Clearly, it is easy to rewrite the main formula for compact sets, just taking the parametrization defined on some (open) neighborhood.

Let us state the main result of the paper.
\begin{thm}\label{thm:main}
Let $U$ and $V$ be two domains in $\mathbb{R}^{n-m}$ and $\mathbb{R}^m$, respectively. Let $f:U\times Y\to\Omega$ be $C^1$--smooth diffeomorphism. Denote by $\Sigma$  the family of $m$--dimensional surfaces $\sigma_x$, $x\in U$, being the images of $V$ with respect to $f$, $\sigma_x=f(x,V)$. Then the extremal function for the $p$--modulus of $\Sigma$ is the following
\begin{equation}\label{eq:mainextremal}
f_{\Sigma}(z)=\frac{1}{l(x)}\left(\frac{|J^y_f|}{|J_f|}\right)^{q-1}\circ f^{-1}(z),\quad z=f(x,y),
\end{equation}
where
\begin{equation}\label{eq:mainhat}
l(x)=\int_V \left(\frac{|J^y_f|}{|J_f|}\right)^q J_f\,dy.
\end{equation}
Moreover, the $p$--modulus of $\Sigma$ equals
\begin{equation}\label{eq:mainmod}
{\rm mod}_p(\Sigma)=\int_U l(x)^{1-p}\,dx.
\end{equation}
\end{thm} 

Let us explain the notation used in the theorem above. The jacobian $|J^y_f|$ is the jacobian of a map $V\ni y\mapsto f(x,y)$ with fixed $x\in U$. Recall, that the jacobian here is a square root of sum of second powers of determinants of maximal minors of the differential.

In our approach, we assume $C^1$--smoothness of $f$. Less restrictive conditions were assumed in \cite{BMV}, namely, that $f$ is in the Sobolev class $W^{1,p}$ with finite distortion. We comment on this in the last section.

Throughout the paper $p$ and $q$ are conjugate coefficients, i.e., $\frac{1}{p}+\frac{1}{q}=1$.

\section{Rodin's formula in any codimension}

\subsection{Necessary facts}
Let us recall the co--area formula and the the formula for the extremal function for the $p$--modulus of a family of level sets of a submersion.

Let $F$ be a $C^1$--diffeomorphism of a domain $\Omega\subset\mathbb{R}^n$ onto $U\subset\mathbb{R}^m$. Denote by $\sigma_x$ the level set $F^{-1}(x)$ and let $H^m_{\sigma_x}$ be $m$--dimensional Hausdorff measure on $\sigma_x$. The following is a well--known fact.

\begin{thm}[Co--area formula]\label{thm:coarea}
For any integrable function $g$ on $\Omega$ the following formula holds
\begin{equation*}
\int_{\Omega} g(z)|J_F(z)|\,dz=\int_U\int_{\sigma_x} g(z)\,dH^m_{\sigma_x}(z)\,dx,
\end{equation*} 
\end{thm}

Denote by $\Sigma$ a family of level sets $\sigma_x$, $x\in U$, or, more precisely, the family of $m$--Hausdorff measures on these level sets. Then the following fact holds \cite{KP} (see also \cite{MCN}).

\begin{thm}[Modulus of level sets]\label{thm:levelsets} 
Assume ${\rm mod}_p(\Sigma)>0$. Then, the extremal function and the $p$--modulus of $\Sigma$ are given, respectively, by
\begin{equation*}
f_{\Sigma}=\frac{|J_F|^{q-1}}{\widehat{|J_F|^{q-1}}\circ F}\quad\textrm{and}\quad {\rm mod}_p(\Sigma)=\int_{U}\left(\widehat{|J_F|^{q-1}}\right)^{1-p}\,dx.
\end{equation*}
\end{thm}
In the statement of above theorem $\hat{f}$ denotes the integral of non--negative function $f$ on the level sets of $F$, i.e., 
\begin{equation*}
\hat{f}(x)=\int_{\sigma_x} f(z)\,dH^m_{\sigma_x}(z),\quad x\in U. 
\end{equation*}
In order to keep the article self contained we recall the proof of Theorem \ref{thm:levelsets}.
\begin{proof}
Clearly, the integral of $f_{\Sigma}$ on any level set equals $1$,
\begin{equation*}
\hat{f_{\Sigma}}=\frac{\widehat{|J_F|^{q-1}}}{\widehat{|J_F|^{q-1}}}\circ F=1.
\end{equation*}
Thus $f_{\Sigma}$ is admissible for the $p$--modulus of $\Sigma$. By H\"older inequality, for any admissible $f$, we have
\begin{align*}
1\geq \int_{\sigma_x} f\,dH^m_{\sigma_x} &=
\int_{\sigma_x}\frac{f}{|J_F|^{\frac{1}{p}}}|J_F|^{\frac{1}{p}}\,dH^m_{\sigma_x}\\
&\leq \left(\int_{\sigma_x}\frac{f^p}{|J_F|}\,dH^m_{\sigma_x}\right)^{\frac{1}{p}}\left(\int_{\sigma_x}|J_F|^{\frac{q}{p}}\,dH^m_{\sigma_x}\right)^{\frac{1}{q}}.  
\end{align*}
Thus
\begin{equation*}
\int_{\sigma_x}\frac{f^p}{|J_F|}\,dH^m_{\sigma_x}\geq\left(\int_{\sigma_x}|J_F|^{q-1}\,dH^m_{\sigma_x}\right)^{1-p}.
\end{equation*}
By the co-area formula (Theorem \ref{thm:coarea}) we get
\begin{align*}
\int_{\Omega} f^p\,dz &=\int_U\int_{\sigma_x}\frac{f^p}{|J_F|}\,dH^m_{\sigma_x}\,dy\\
&\geq\int_U\widehat{|J_F|^{q-1}}^{1-p}\,dx\\
&=\int_U\int_{\sigma_x}\frac{|J_F|^{q-1}}{\widehat{|J_F|^{q-1}}^p}\,dH^m_{\sigma_x}\,dx\\
&=\int_{\Omega}\frac{|J_F|^q}{\widehat{|J_F|^{q-1}}^p}\,dz\\
&=\int_{\Omega}f_{\Sigma}^p\,dz.
\end{align*}
Thus $f_{\Sigma}$ is extremal and the formula for ${\rm mod}_p(\Sigma)$ holds.
\end{proof}

\subsection{An Algebraic lemma}
In this subsection we will prove the key algebraic fact, that will be used in the proof of the main theorem of the article. It is necessary for the relation between jacobians of appropriate mappings.

Denote by $\mathcal{M}(n,m)$ the space of matrices consisting of $n$ rows and $m$ columns. For a matrix $A\in\mathcal{M}(n,m)$ denote by $|A|^2$ a sum of determinants of all maximal rank minors. In particular, if $A\in\mathcal{M}(n,n)$, then $|A|=\det A$ and if $A\in\mathcal{M}(n,m)$ with $n>m$, then $|A|^2=\det(A^{\top}A)$.

The following lemma is a generalization of Lemma \ref{lem:algebraic1} to arbitrary dimension. We could try to adapt the proof to this case by applying successively Laplace expansion, however, it requires a lot of computations. We present, in our opinion, more elegant and brief proof. 

\begin{lem}\label{lem:main}
Let $A\in GL(n,\mathbb{R})$ and $B\in\mathcal{M}(n-m,n)$. If $B\cdot A=(I_{n-m}\, 0)$, then
\begin{equation*}
|A'|=|A||B|,
\end{equation*}
where $A'$ is made of last $m$ columns of $A$.
\end{lem}
\begin{proof}  
Split matrix $A$ in the following way $A=(A_0\,A')$, where $A_0\in\mathcal{M}(n,n-m)$. The assumption is equivalent to saying that $BA_0=I_{n-m}$ and $BA'=0$. Consider a matrix of the form $X=(B^{\top}\, A')\in\mathcal{M}(n,n)$. The subspaces spanned by first $n-m$ and last $m$ columns of $X$ are, by assumption, orthogonal. Thus
\begin{equation}\label{eq:lemmaproof}
|X|=|B^{\top}||A'|=|B||A'|.   
\end{equation}
Let us consider the matrix $A^{\top}X\in\mathcal{M}(n,n)$. We can express it, by the assumption, in the following form
\begin{equation*}
A^{\top}X=(A^{\top}B^{\top}\, A^{\top}A')=((BA)^{\top}\,A^{\top}A')
=\left(\begin{array}{cc} I_{n-m} & A_0^{\top}A' \\ 0 & (A')^{\top}A' \end{array}\right).
\end{equation*}
Hence, $|A||X|=|(A')^{\top}A'|=|A'|^2$. This, together with \eqref{eq:lemmaproof}, implies $|A||B|=|A'|$.
\end{proof}

\subsection{Proof of Rodin's formula in any codimension}

Let $F:\Omega\to U$ be of the form $F=\pi\circ f^{-1}$, where $\pi:U\times V\to U$ is a projection on $U$. Then $F$ is a submersion with jacobian $J_F$. Differentiating the condition $F\circ f=\pi$ we get
\begin{equation*}
(DF\circ f)\cdot Df=(I_{n-m}\,0). 
\end{equation*}
Applying Lemma \ref{lem:main} to $A=Df$ and $B=DF\circ f$ (then $A'=D_yf$), we get
\begin{equation*}
|J^y_f|=|J_f|(|J_F|\circ f).
\end{equation*}
Thus the formula for $l(x)$ takes the form
\begin{equation*}
l(x)=\int_U \left(\frac{|J^y_f|}{|J_f|}\right)^{q-1}|J^y_f|\,dx
=\int_{\sigma_x}|J_F|^{q-1}\,dH^m_{\sigma_x}=\widehat{|J_F|^{q-1}}(x).
\end{equation*}
Thus, by Theorem \ref{thm:levelsets} we have
\begin{equation*}
f_{\Sigma}(z)=\frac{1}{l(x)}\left(\frac{|J^y_f|}{|J_f|}\right)^{q-1}\circ f^{-1}(z),\quad f(x,y)=z,
\end{equation*}
and ${\rm mod}_p(\Sigma)=\int_U l(x)^{1-p}\,dx.$

\section{Some Applications}

In this section we give some applications of obtained formulae \eqref{eq:mainextremal} and \eqref{eq:mainmod}. We consider generalizations of examples studied in \cite{BMV} except for the case of conjugate submersions (Example \ref{exa:conjugate} below).

\begin{exa}[Generalized condenser]\label{exa:condenser}
In this example, we will discus the formula obtained in \cite{BMV} for a curve (or surface) family in a condenser defined by some function $u$. Namely, let $u:U\times V\to\mathbb{R}^n$ be a $C^1$--diffeomorphism, where $U$ and $V$ are domains in $\mathbb{R}^{n-m}$ and $\mathbb{R}^n$, respectively. Denote by $\Sigma_0$ the family of all $m$--dimensional surfaces $S_x$, $x\in U$, in $\mathbb{R}^n$ of the form $S_x=\{u(x,y)\mid y\in V\}$. Moreover, let $f$ be a $C^1$--diffeomorphism of $u(U\times V)$ and let $\Sigma=f(\Sigma_0)$, meaning that $\Sigma$ is a family of surfaces $\sigma_x=f(S_x)$. Applying Theorem \ref{thm:main} to the map $u\circ f$ we get the following formulae for the extremal function and the $p$--modulus of $\Sigma$.
\begin{cor}\label{cor:generalizedcondenser}
The extremal function $f_{\Sigma}$ and the $p$--modulus ${\rm mod}_p(\Sigma)$ of the family $\Sigma$ are given, respectively, by
\begin{align*}
f_{\Sigma}(z) &=\frac{1}{l(x)}\left(\frac{|I^y_f|}{|I_f|}\right)^{q-1}\circ f^{-1}(z), \quad z=(f\circ u)(x,y),\\
{\rm mod}_p(\Sigma_0) &=\int_U l(x)^{1-p}\,dx,
\end{align*}
where $I_f=J_{f\circ u}=(J_f\circ u)J_u$ and
\begin{equation*}
l(x)=\int_V \left(\frac{|I^y_f|}{|I_f|}\right)^q |I_f|\, dz.
\end{equation*}
\end{cor}
\end{exa}

\begin{exa}[Parallel surfaces]\label{exa:parallel}
Consider a product $U\times V$, where $U$ and $V$ are domains in $\mathbb{R}^{n-m}$ and $\mathbb{R}^m$, respectively. Denote by $\Sigma$ the family of surfaces of the form $\{x\}\times V$, $x\in U$. Family $\Sigma$ is defined by the identity map $f={\rm id}$ on $U\times V$. Thus by Theorem \ref{thm:main} the extremal function $f_{\Sigma}$ and  the $p$--modulus of $\Sigma$ are, respectively,
\begin{equation*}
f_{\Sigma}=H^m(V)^{-1}\quad\textrm{and}\quad {\rm mod}_p(\Sigma)=H^{n-m}(U)H^m(V)^{1-p}.
\end{equation*}
In particular, the extremal function is constant. Analogously, considering a family $T$ of 'horizontal' surfaces $U\times\{y\}$, $y\in V$, we have
\begin{equation*}
f_T=H^{n-m}(U)^{-1}\quad\textrm{and}\quad {\rm mod}_p(T)=H^{n-m}(U)^{1-q}H^m(V).
\end{equation*}
In particular,
\begin{equation}\label{eq:productmoduli}
{\rm mod}_p(\Sigma)^{\frac{1}{p}}{\rm mod}_q(T)^{\frac{1}{q}}=1.
\end{equation}
\end{exa}

In the example below we force the condition \eqref{eq:productmoduli} by requiring a relation on restricted jacobians of $f$.

\begin{exa}[Conjugate submersions]\label{exa:conjugate}
In \cite{CP} the authors, generalizing the notion of conformality, introduced, so called, conjugate submersions. These define two orthogonal foliations such that the product of $p$-- and $q$--modulus of them equals one. Treated together they define a diffeomorphism from the manifold to the product of two manifolds. Namely, we say that two mappings $\varphi:M\to R$ and $\psi:M\to S$ between Riemannian manifolds are $(p,q)$--conjugate if $(\varphi,\psi):M\to R\times S$ is a diffeomorphism, distributions ${\rm ker}D\varphi$ and ${\rm ker}D\psi$ are orthogonal and, finally, $|J_{\varphi}|^p=|J_{\psi}|^q$.

Let us reverse this approach (with slight modifications including switching coefficients $p$ and $q$), by considering a (global) parametrization of a domain by a product of two domains. Consider a $C^1$--diffeomorphism $f:U\times V\to\Omega$, where $U$ and $V$ are domains in $\mathbb{R}^{n-m}$ and $\mathbb{R}^m$, respectively, and $\Omega$ is a domain in $\mathbb{R}^n$. Fix two conjugate coefficients $(p,q)$. We say that $f$ is $(p,q)$--{\it conformal} if
\begin{equation*} 
\frac{|J^x_f|^p}{|J_f|^p}=\frac{|J^y_f|^q}{|J_f|^q}\quad\textrm{and} \quad D_xf\perp D_yf.
\end{equation*}
The second condition means that tangent spaces to $f(U,y)$ and $f(x,V)$ are orthogonal for any $(x,y)$. It implies $J^x_f J^y_f=J_f$. Thus, by the first condition we have
\begin{equation}\label{eq:pqmap}
|J^y_f|^p=|J^x_f|^q=|J_f|.
\end{equation} 
We use above condition as a definition of a $(p,q)$--{\it map}. 

Assume, therefore, that $f$ is a $(p,q)$--map. Denote by $\Sigma$ and $T$ the families of $m$ and $(m-n)$--dimensional surfaces defined by
\begin{equation*}
\Sigma=\{\sigma_x=f(x,V)\}_{x\in U},\quad T=\{\tau_y=f(U,y)\}_{y\in V},
\end{equation*}
respectively. By \eqref{eq:pqmap} we have $l_{\Sigma}(x)=H^m(V)$ and $l_T(y)=H^{n-m}(U)$. Thus by Theorem \ref{thm:main}
\begin{equation*}
{\rm mod}_p(\Sigma)=H^{n-m}(U)H^m(V)^{1-p},\quad {\rm mod}_q(T)=H^m(V)H^{n-m}(U)^{1-q}.
\end{equation*}
In particular, ${\rm mod}_p(\Sigma)^{\frac{1}{p}}{\rm mod}_q(T)^{\frac{1}{q}}=1$.
\end{exa}

\begin{exa}[General shear transformation]\label{exa:shear}
Consider a parallel 'surfaces' form the Example \ref{exa:parallel} and let $f:\mathbb{R}^n\to\mathbb{R}^n$ be a general shear transformation:
\begin{equation*}
f(x,y)=(x+By,y),\quad x\in\mathbb{R}^m,y\in\mathbb{R}^{n-m},
\end{equation*}
where $B\in\mathcal{M}(n,n-m)$. In matrix notation, $f$ is a linear transformation defined by a matrix $M_f=\left(\begin{array}{cc} I_m & B \\ 0 & I_{n-m} \end{array}\right)$. We will compute the $p$--modulus of $\Sigma=f(\Sigma_0)$. Clearly, $I_f=\det M_f=1$ and $|I^y_f|=(\det(B^{\top}B+I_{n-m}))^{\frac{1}{2}}$. Thus, applying Corollary \ref{cor:generalizedcondenser} to the family $\Sigma=f(\Sigma_0)$, we have
\begin{equation*}
f_{\Sigma}=H^m(V)^{-1}(\det(B^{\top}B+I_{n-m}))^{-\frac{1}{2}}
\end{equation*}
and
\begin{equation*} 
{\rm mod}_p(\Sigma)=H^{n-m}(U)H^m(V)^{1-p}(\det(B^{\top}B+I_{n-m}))^{-\frac{p}{2}},
\end{equation*}
since $l=H^m(V)\det^{\frac{q}{2}}(B^{\top}B+I_{n-m})$.   
\end{exa}

\section{Final remarks}

In this section we comment on possible weakest conditions which can be imposed on $f$ in Theorem \ref{thm:main}. In \cite{BMV}, i the case of curve family, the authors assume that $f$ is in Sobolev class $W^{1,p}$. Then $f$ is $ACL^p$ and, hence, for $H^{n-1}$--almost all $x\in D$, the derivative $\frac{\partial f}{\partial t}$ exists ($t$ is a parameter for curves in the considered family). In the analogous situation but for 'hypersurface' family, the authors of \cite{BMV} assume the parametrization of surfaces is $C^1$--smooth and they deform the surface family by a $W^{1,p}$--homeomorphism with finite distortion. This allows to deduce that the inverse map $f^{-1}$ is $W^{1,1}$ with finite distortion. In the formula for a $p$--modulus of this family, the gradient of submersion defining the hypersurfaces is used, while for a map $f$ the only information needed is the jacobian $J_f$. The authors rely on area and coarea formulea in the possible weakest forms (see \cite{MSZ}). 
 
We feel that the case of $C^1$--smoothness in not so restrictive, although, probably, may be slightly weakened.

\end{document}